\documentclass[a4paper,12pt,intlimits,oneside]{amsart}

\usepackage[a4paper, total={6in, 9in}]{geometry}
\usepackage[T1]{fontenc}
\usepackage[utf8]{inputenc}
\usepackage{comment}
\usepackage{amsfonts}
\usepackage{amsmath}
\usepackage{amsthm}

\let\oldcap\cap
\usepackage{mathabx}
\let\cap\oldcap

\usepackage[
    backend=biber, 
    natbib=true,
    style=numeric,
    sorting=none
]{biblatex}

\usepackage{pgfplots}
\usepackage{tikz}
\usepackage{mathabx}
\usepackage{comment}
\usepackage{graphicx,subcaption}
\usepackage{wrapfig}
\usepackage{tabstackengine}
\usepackage{nicematrix}

\usepackage{mathtools}
\usepackage{bbm}

\newcounter{blockcount} 
\newcommand{\blockelement}[1]{%
    \stepcounter{blockcount}
    \ifnum\value{blockcount}>1
        & 0 & 0 & \cdots & 0 \\ 
    \fi
    #1 & 0 \\ 0 & #1
    \ifnum\value{blockcount}<\numexpr\ProcessListLength/2
        \\ 0 & 0 & \cdots & 0 &
    \fi
}

\usetikzlibrary{positioning} 
\usetikzlibrary{calc}
\usetikzlibrary{arrows,shapes,backgrounds}
\usepackage{xparse}

\usepackage{mathtools}
\usepackage{esint}

\newtheorem{theorem}{Theorem}[section]

\newtheorem{lemma}[theorem]{Lemma}

\usepackage{bbm}

\title{Half-Wave Maps: Explicit Formulas for Rational Functions with Simple Poles}

\author{Gaspard Ohlmann}

\date{}

\newcommand{\Diag}[3]{%
    \begin{pmatrix}
        #1 & 0 & \cdots & 0 \\
        0 & #2 & \ddots & \vdots \\
        \vdots & \ddots & \ddots & 0 \\
        0 & \cdots & 0 & #3
    \end{pmatrix}
}

\begin{document}

\newcommand{\BB}{\mathcal{B}}
\newcommand{\LL}{\mathcal{L}}
\newcommand{\II}{I_{2\times 2}}
\newcommand*{\dd}{\mathop{}\!\mathrm{d}}
\newcommand{\iu}{{i\mkern1mu}}
\newcommand{\im}{{\operatorname{Im}}}

\maketitle

\begin{abstract}
    We establish an explicit formula for the Half-Wave maps equation for rational functions with simple poles. The Lax pair provides a description of the evolution of the poles. By considering a half-spin formulation, we use linear algebra to derive a time evolution equation followed by the half-spins, in the moving frame provided by the Lax pair. We then rewrite this formula using a Toeplitz operator and $G$, the adjoint of the operator of multiplication by $x$ on the Hardy space $L_+^2(\mathbb{R})$.
\end{abstract}


\tableofcontents{}

\section{Introduction}

\subsection{The Half-Wave Maps equation}

The Half-Wave Maps (HWM) equation is a nonlinear partial differential equation that models the evolution of wave functions mapping into the sphere $\mathbb{S}^2$. It is given by 
\begin{equation}\tag{HWM}\label{HWM}
\partial_t m(t,x) = m(t,x) \times |\nabla| m(t,x), 
\end{equation}
where $m: \mathbb{R} \times \mathbb{R} \to \mathbb{S}^2$ represents the spin field, and $|\nabla|$ denotes the (pseudo-differential) operator corresponding to the multiplication on the Fourier side $\widehat{|\nabla| f}(\xi) = |\xi| \hat{f}(\xi)$.

Equation \eqref{HWM} formally preserves the energy  

\begin{equation*}
E[m]=\frac{1}{2} \int_{\mathbb{R}} m \cdot|\nabla| m d x=c \iint_{\mathbb{R} \times \mathbb{R}} \frac{|m(x)-m(y)|^2}{|x-y|^2} d x d y,
\end{equation*}
and we consider here functions in the corresponding energy space, i.e. $m(t,\cdot) \in \dot H^{1/2}(\mathbb{R})$, or equivalently $|\xi|^{1/2} \hat m(t,\xi) \in L_\xi^2 (\mathbb{R})$.

The Half-wave maps equation, both in the general and rational case, has recently caught interest due to its rich mathematical structure and connections to integrable systems. For a recent survey on the (HWM) equation, we refer the reader to \cite{lenzmann2018short}. Notably, it arises as a continuum limit of a classical version of the Haldane-Shastry spin chain, as shown by Zhou and Sone in \cite{zhou2015solitons}, and arises as an effective equation in the continuum limit of completely integrable Calogero-Moser classical spin systems with inverse square $1/r^2$, as demonstrated by Lenzmann and Sok \cite{lenzmann2020derivation}. 
Moreover, it emerges as a limit case of a spin generalization of the Benjamin-Ono equation (sBO), as shown in \cite{berntson2022spin}.

It is conjectured that a formula similar to the one found by P. Gerard in \cite{gerard2023explicit} for the Benjamin-Ono equation exists for this equation. In this article, we show the existence of such a formula in the rational case with simple poles. This formula facilitates the analysis of various problems, including the global-in-time existence in the rational case with simple poles or the description of asymptotic behaviors.

In higher dimensions, J. Krieger and Y. Sire show in \cite{krieger2017small} that the Half-Wave Maps equation is well posed for $d\geq 5$ for small initial data. This  result has been improved in \cite{kiesenhofer2021small} by J. Krieger and A. Kiesenhofer to $d \geq 4$. The methods used rely on Strichartz estimates that do not hold in spatial dimension one, as the problem is energy critical for $d=1$. 
 It has been shown in \cite{zhou2015solitons} and \cite{gerard2018lax} that the Half-Wave maps equation exhibits an infinite number of preserved quantities, such as mass, energy, and momentum. In \cite{lenzmann2018energy}, E. Lenzmann and A. Schikorra completely classify all the traveling waves. They can be described by the choice of their poles, a Blaschke product, and a rotation. Finally, the inverse scattering transform (IST) method has been used to the local nonlinear evolution equations such as the (KdV) and (NLS) equations, but its success is reduced for the (HWM) equation for which the Lax pair takes a nonlocal form. See \cite{22,23,24,25} for the (BO) and nonlocal (NLS) equation.

We introduce rational solutions, i.e. functions $m$ satisfying equation \eqref{HWM} that can be put in the form
\begin{equation}\label{PA}
    m(t,x) = m_0 + \sum_{j=1}^N \frac{s_j(t)}{x-x_j(t)} + \sum_{j=1}^N \frac{\bar s_j(t)}{x-\bar x_j(t)},~\im(x_j)>0
\end{equation}
where $s_j(t)$, the spins, are complex vectors in $\mathbb{C}^3$ and $x_j(t)$, the poles, are elements of $\mathbb{C}$. In \cite{berntson2020multi}, the authors show that \eqref{HWM} is satisfied if and only if the poles and spins satisfy the time evolution equations
\begin{equation*}
    \dot s_j(t) = 2 i \sum_{j\neq k}^N \frac{s_j(t) \times s_k(t)}{(x_j(t)-x_k(t))^2}, \quad j=1,2, \ldots, N,
\end{equation*}
\begin{equation*}
    \ddot x_j(t) = -4 \sum_{j \neq k} \frac{s_j(t) \cdot s_k(t)}{(x_j(t)-x_k(t))^3}, \quad j=1,2, \ldots, N,
\end{equation*}
and satisfy the the additional constraints

$$
\dot{x}_j(0)=\frac{\mathbf{s}_{j, 0} \times \mathbf{s}_{j, 0}^*}{\mathbf{s}_{j, 0} \cdot \mathbf{s}_{j, 0}^*} \cdot\left(\mathrm{im}_0+i \sum_{k \neq j}^N \frac{\mathbf{s}_{k, 0}}{x_{j, 0}-x_{k, 0}}+i \sum_{k=1}^N \frac{\mathbf{s}_{k, 0}^*}{x_{j, 0}-x_{k, 0}^*}\right), \quad j=1,2, \ldots, N ,
$$

$$
\mathbf{s}_{j, 0}^2=0, \quad \mathbf{s}_{j, 0} \cdot\left(\mathrm{im}_0+i\sum_{k \neq j}^N \frac{\mathbf{s}_{k, 0}}{x_{j, 0}-x_{k, 0}}+i\sum_{k=1}^N \frac{\mathbf{s}_{k, 0}^*}{x_{j, 0}-x_{k, 0}^*}\right)=0, \quad j=1,2, \ldots, N .
$$
The solution is thus characterized by the spins $s_j$ and the poles $x_j$ that evolve according to the dynamics of an exactly solvable spin CM system, studied for instance in \cite{calo1}, \cite{calo2} or \cite{calo3}, with constraints.

In \cite{ohlmann2024studywellposedness1denergycritical}, the author establishes in the rational case with simple poles, the well-posedness of the equation for finite and infinite times in several cases, for instance assuming that the spins stay bounded. The explicit formula allows for more control over the spins. Using the formula, we are currently working on removing the assumptions for the well-posedness as well as providing additional properties on the long-term behavior. 

We now present the works establishing a Lax pair for (HWM) that we will use extensively. In \cite{matsuno2022integrability}, Matsuno defines the following two matrices

\begin{equation*}
    L_{i,j}(t) = \left\{
    \begin{aligned}
    & \dot x_i(t),~ i=j,\\
    &\varepsilon_{i,j} \frac{\sqrt{ - 2 s_i(t)\cdot s_j(t)}}{(x_i(t) - x_j(t))},~ i\neq j,
    \end{aligned}
    \right.
    \quad
    B_{i,j}(t) = \left\{
    \begin{aligned}
    & 0,~ i=j,\\
    &\varepsilon_{i,j} \frac{\sqrt{ - 2 s_i(t)\cdot s_j(t)}}{(x_i(t) - x_j(t))^2},~ i\neq j.
    \end{aligned}
    \right.
\end{equation*}
He then shows that the matrix $L$ satisfies the key Lax equation 

\begin{equation}\label{LaxMatsuno}
    \partial_t L(t) = [B(t),L(t)] = B(t) L(t) - L(t) B(t).
\end{equation}
In particular, he obtains the corollaries, with $X(t)$ being the diagonal $(x_1(t),\dots,x_N(t))$ 

\begin{equation*}
    \dot X(t) = L(t) + [B(t),X(t)],
\end{equation*}
and with $U(t)$ being the solution of 

\begin{equation*}
    \left\{
        \begin{aligned}
            &U(0) = I_N,\\
            &\dot U(t) = B(t) U(t),
        \end{aligned}
    \right.
\end{equation*}
then the matrices $L(t)$, $S(t)$ and $X(t)$ evolve according to the time evolution equations

\begin{equation*}
    \left\{
        \begin{aligned}
            &L(t) = U(t) L(0) U(t)^{-1},\\
            &S(t) = U(t) S(0) U(t)^{-1},\\
            &X(t) = U(t) \left( X(0) + t L(0) \right) U(t)^{-1}.
        \end{aligned}
    \right.
\end{equation*}

Hence, in the moving base $U(t)$, the evolution of the poles can be described easily. This works follows the work of Lenzmann and Gerard \cite{gerard2018lax} in which a more general Lax-pair formulation is established for this equation. We introduce the $2\times 2$ matrix formulation, an equivalent representation of (HWM) considered in \cite{gerard2018lax}, inspired by the study of the classical Heisenberg model in \cite{takhtajan1977integration}. With $\sigma_1, \sigma_2, \sigma_3 \in \mathfrak{s u}(2)$ the standard Pauli matrices, and $s\in \mathbb{C}^3$ a complex vector representing a spin, we associate the matrix 

\begin{equation*}
A= s \cdot \boldsymbol{\sigma}=\sum_{j=1}^3 s_j \sigma_j=\left(\begin{array}{cc}
s_3 & s_1-\mathrm{i} s_2 \\
s_1+\mathrm{i} s_2 & -s_3
\end{array}\right) .
\end{equation*}
With $A_j(t) = s_j(t) \cdot \boldsymbol{\sigma}$, $M_0 = m_0 \cdot \boldsymbol{\sigma}$, and $M(t,x) = m(t,x) \cdot \boldsymbol{\sigma}$, then by linearity, 
\begin{equation}\label{Mmat}
    M(t,x) = M_0 + \sum_{j=1}^N \frac{A_j(t)}{x-x_j(t)} + \sum_{j=1}^N \frac{A_j^*(t)}{x-\bar x_j(t)}.
\end{equation}
For this new representation, the (HWM) equation now reads

\begin{equation*}
    \partial_t M(t,x) = -\frac{i}{2} [M(t,x), |\nabla| M(t,x)],
\end{equation*}
where we used the identity $[s \cdot \boldsymbol{\sigma},t \cdot \boldsymbol{\sigma}] = 2 i (s \wedge v) \cdot \boldsymbol{\sigma} $. Using the same identity, the time evolution for the spins, now represented by $2\times 2$ matrices, is given by 
\begin{equation*}
    \partial_t A_j(t) = \sum_{k \neq j} \frac{[A_j(t),A_k(t)]}{(x_j(t)-x_k(t))^2}.
\end{equation*}
Finally, using that for $X,Y \in \mathbb{R}^3$, 

\begin{equation*}
(X \cdot \boldsymbol{\sigma})(Y \cdot \boldsymbol{\sigma})=(X \cdot Y) I_{2\times 2}+\mathrm{i}(X \wedge Y) \cdot \boldsymbol{\sigma},
\end{equation*}
the constraints now take the form

\begin{equation}\label{consmat}
\left\{
    \begin{aligned}
        &U_{\infty}^2=1,~U_\infty^* = U_\infty,~Tr(U_\infty)=0,\\ 
        &A_j^2=0,~B_j A_j + A_j B_j =0,
    \end{aligned}
    \right.
\end{equation}
with
\begin{equation*}
    B_j= U_\infty + \frac{A_j^*}{x_j - \bar x_j} + \sum_{k\neq j} \left( \frac{A_k^*}{x_j - \bar x_k} + \frac{A_k}{ x_j -  x_k} \right).
\end{equation*}
In \cite{gerard2018lax}, Lenzmann and Gérard established the following Lax pair:

\begin{equation}\label{LaxEnno}
    \partial_t \mathcal{L}_S = [\mathcal{B}_S,\mathcal{L}_S],
\end{equation}
where $\mu_S$ is the multiplication by $S$ operator, $\mathcal{L}_S= [H,\mu_S]$ and $\mathcal{B}_S = -\frac{i}{2} (\mu_S |\nabla| +|\nabla| \mu_S) + \frac{i}{2} \mu_{|\nabla| S}$.

It is noteworthy that expressing \eqref{LaxEnno} in a soliton basis yields a relation that closely resembles, yet differs from, equation \eqref{LaxMatsuno}. Additionally, the associated conserved quantities vary between the two formulations.

\subsection{Statement of the main result}

We now introduce the notations that we need to state the main theorem of this paper.

We denote by $L_{+}^2(\mathbb{R})$ the Hardy space corresponding to functions in $L^2(\mathbb{R})$ having a Fourier transform supported in the domain $\xi \geq 0$. The space $L^2(\mathbb{R})$ identifies to holomorphic functions $f$ on the upper half plane $\mathbb{C}_+ = \{ z \in \mathbb{C},~\im(z)>0 \}$ such that 

\begin{equation*}
    \sup_{y >0} \int_{\mathbb{R}} |f(x+ iy)|^2 dx < \infty.
\end{equation*}
We denote by $\Pi_+$ the orthogonal projector from $L^2(\mathbb{R})$ onto $L_+^2(\mathbb{R})$, and define the Toeplitz operators $T_{g}$, for $g \in L^\infty(\mathbb{R})$, as $T_g(f) = \Pi_+ (fg)$ for $f \in L_+^2(\mathbb{R})$. Similarly, we define $\Pi_-$ as the projection from $L^2(\mathbb{R})$ onto $L_-^2(\mathbb{R})$, the space of functions in $L^2(\mathbb{R})$ having a Fourier transform supported in the domain $\xi \leq 0$. Finally, for $f \in L_2^+(\mathbb{R})\cap H^1(\mathbb{R})$, we define $G$ and $I_+$ as
\begin{equation*}
    \widehat{Gf}(\xi) = i \frac{\partial}{\partial \xi} \left[ \hat f(\xi) \right] \mathbbm{1}_{\xi>0},\quad I_+(f)= \hat f(0^+).
\end{equation*}

\begin{theorem}\label{gerardformula}
Let
\begin{equation*}
    M(t,x) = M_0 + V(t,x)
\end{equation*}
be a rational function with simple poles of the form 
\begin{equation*}
    V(t,x) = \sum_{j=1}^N \frac{A_j(t)}{x-x_j(t)} + \sum_{j=1}^N \frac{A_j^*(t)}{x-\bar x_j(t)},~\im(x_j)>0,
\end{equation*}
satisfying the Cauchy problem
\begin{equation*}
    \left\{
    \begin{aligned}
        &\partial_t M(t,x) = -\frac{i}{2} [M(t,x),|\nabla| M(t,x)],\\
        &V(0,x) = V_0(x).
    \end{aligned}
    \right.
\end{equation*}
Then, $\Pi_+ V(t,x)$ is given by (with $U_0(x) = U(0,x)$)
\begin{equation*}
    \Pi_+ V(t,x) = \frac{1}{2i \pi} \cdot I_+ \left[ (G-t T_{U_0} - x Id)^{-1} \Pi_+ V_0 \right].
\end{equation*}

\end{theorem}

Note that the explicit formula we present bears resemblance to the one established for the Benjamin-Ono equation in \cite{gerard2023explicit}, which is given by

\begin{equation*}
    \Pi_+ u = \frac{1}{2i \pi } I_+ [(G-2 t L_{u_0} - z I_N)^{-1} \Pi_+ u_0].
\end{equation*}

However, the methodology underlying our proof differs substantially from that of \cite{gerard2023explicit}. We first prove an equivalent result, involving constant matrices given by the initial condition at $t=0$. To do so, we introduce the half-spins formulation and find a time evolution equation for the half-spins, that we solve.
Considering then the corresponding Lax pair, we take advantage of the structure provided for the poles that, once combined with the evolution of the half-spins, leads to a simple description of the evolution. 
This formulation inherently requires that the function is rational with simple poles, but coincide in that case with the more abstract formula given by theorem \ref{gerardformula}.
Extending the representation should permit to extend the result to the case where multiple poles are involved. However, extending the formula to encompass general solution requires additional arguments beyond the scope of the present work.

We now state the explicit formula in its matrix form, where the constant matrices $T$ and $\mathcal{H}$ will be given in the next section.

\begin{theorem}\label{HST}
Let
\begin{equation*}
    M(t,x) = M_0 + V(t,x)
\end{equation*}
be a rational function with simple poles of the form
\begin{equation*}
    V(t,x) = \sum_{j=1}^N \frac{A_j(t)}{x-x_j(t)} + \sum_{j=1}^N \frac{A_j^*(t)}{x-\bar x_j(t)},~\im(x_j)>0,
\end{equation*}
satisfying the Cauchy problem
\begin{equation*}
    \left\{
    \begin{aligned}
        &\partial_t M(t,x) = -\frac{i}{2} [M(t,x),|\nabla| M(t,x)],\\
        &V(0,x) = V_0(x).
    \end{aligned}
    \right.
\end{equation*}
Then, for 
\[
\alpha_j = \sqrt{-s_{j,1} + i s_{j,2}},~\beta_j = \sqrt{s_{j,1} + i s_{j,2}},
\]
there exists two constant matrices $T \in \mathbb{R}^{2N\times 2}$, $\mathcal{H}\in \mathbb{R}^{2N\times 2N}$, such that with the two diagonal matrices
\[
\mathcal{E},\mathcal{F} \in \mathbb{C}^{2N\times 2N},~
\mathcal{E}_{i,j} = \left\{
\begin{aligned}
&\delta_{i,j} \alpha_j,~ j\text{ odd}\\
&\delta_{i,j} \beta_j,~j\text{ even}
\end{aligned}
\right. ,\quad
\mathcal{F}_{i,j} = \left\{
\begin{aligned}
&\delta_{i,j} \beta_j,~ j\text{ odd}\\
&\delta_{i,j} (-\alpha_j),~j\text{ even}
\end{aligned}
\right.
\]
$\Pi_- V(t,x)$ is then given by

\begin{equation*}
    \Pi_{-} V(t,x) = -T^T \mathcal{E}(0) \mathcal{H} [X(0) + tL(0) - x I_N]^{-1} \mathcal{F}(0) T.
\end{equation*}

\end{theorem}

\section{Derivation of the explicit formulas}

\subsection{Half-spin formulation and dynamics}

In this section, we introduce the Half-spins formulation. We first show that the dynamics of the Half-spins is governed by a simple equation. Later on, we provide formulas for the matrices $L$ and $B$ that constitute a Lax pair for (HWM). 

Let $M$ be a solution of the Half-Wave Maps equation in the Pauli's matrices form \eqref{Mmat}. Since $A_j^2=0$ and $Tr(A_j) = 0$, the constraints \eqref{consmat} imply the existence of $e_j,~\xi_j \in \mathbb{C}^2$ such that 
\begin{equation*}
    A_j v = e_j (v \cdot \xi_j) = e_j \xi_j^T v,~ e_j\cdot \xi_j =0.
\end{equation*}
The choice of $\xi_j$ and $e_j$ is not unique. This is equivalent to the existence of two diagonal matrices  $E_j$, $F_j$ and such that 
\begin{equation}\label{Hspin}
    A_j = E_j H F_j = \begin{pmatrix}
        e_j[1] & 0 \\
        0 & e_j[2] \\
    \end{pmatrix}
    \begin{pmatrix}
        1& 1 \\
        1 & 1 \\
    \end{pmatrix}
    \begin{pmatrix}
        \xi_j[1] & 0 \\
        0 & \xi_j[2] \\
    \end{pmatrix}, H E_j F_j H = 0.
\end{equation}

We now define the Half-Spin matrices $\mathcal{E}$ and $\mathcal{F}$ that are just diagonal matrices with $E_j$ and $F_j$ on the diagonal

\begin{equation}\label{Erond}
    \mathcal{E}(t) = \begin{pmatrix}
        E_1(t) & 0_{2\times 2} & \dots & 0_{2\times 2}\\
        0_{2\times 2} & E_2(t) & \dots & \vdots \\
        \vdots & \dots & \ddots & \vdots \\
        0_{2\times 2} & \dots & \dots & E_N(t)
    \end{pmatrix} \in (\mathbb{C}^{2\times 2})^{N\times N},
\end{equation}

\begin{equation}\label{Frond}
    \mathcal{F}(t) = \begin{pmatrix}
        F_1(t) & 0_{2\times 2} & \dots & 0_{2\times 2}\\
        0_{2\times 2} & F_2(t) & \dots & \vdots \\
        \vdots & \dots & \ddots & \vdots \\
        0_{2\times 2} & \dots & \dots & F_N(t)
    \end{pmatrix} \in (\mathbb{C}^{2\times 2})^{N\times N}.
\end{equation}
We also denote $\mathcal{E}_0 = \mathcal{E}(0)$ and $\mathcal{F}_0 = \mathcal{F}(0)$. We say that $\mathcal{E}$ and $\mathcal{F}$ are \textbf{Half-spin} matrices representing $V$ if for any $j$, equation \eqref{Hspin} is satisfied.
Finally, we define the two constant matrices, $T = \mathbb{C}^{2N\times 2}$ and $\mathcal{H}\in \mathbb{C}^{2N \times 2N}$ as 

\begin{equation}\label{constants}
    \mathcal{H}(t) = \begin{pmatrix}
        H & 0_{2 \times 2} & \dots & 0_{2 \times 2}\\
        0_{2\times 2} & H & \dots & 0_{2\times 2} \\
        \dots & \dots & \dots & \dots \\
        0_{2 \times 2} & 0_{2\times 2} & \dots & H 
    \end{pmatrix},~ T= \begin{pmatrix}
        I_{2\times 2} \\
        I_{2\times 2} \\
        \vdots \\
        I_{2\times 2}
    \end{pmatrix} = \begin{pmatrix}
        1 & 0 \\
        0 & 1 \\
        \vdots & \vdots \\
        \vdots & \vdots \\
        1 & 0 \\
        0 & 1 
    \end{pmatrix}.
\end{equation}

For a matrix $U \in \mathbb{C}^{N\times N}$, we define $[U] \in \mathbb{C}^{2N \times 2N}$, the doubled matrix, constituted of $N^2$ diagonal blocks of the form
\begin{equation}\label{double}
[U] = \begin{pmatrix}
    U_{1,1} I_{2\times 2} & U_{1,2} I_{2 \times 2} & \dots & U_{1,N} I_{2\times 2} \\
    U_{2,1} I_{2\times 2} & U_{2,2} \II & \dots & U_{2,N} \II \\
    \vdots & \dots & \ddots & \vdots \\
    U_{N,1} \II & \dots  & \dots & U_{N,N} \II
\end{pmatrix}
\end{equation}

Note that for a $2n \times 2 $ matrix, or equivalently a column vector of $2 \times 2$ matrices of the form 
\begin{equation*}
    M = (M_1,\dots,M_N)^T,
\end{equation*}
then the image of $M$ by a doubled matrix $[K] \in \mathbb{C}^{2N \times 2N}$ is a $2 \times 2$ column vector of the form
\begin{equation*}
    [K] M = (M'_1,\dots,M'_N),
\end{equation*}
with 
\begin{equation*}
    M'_j = \sum_{k=1}^N K_{j,k} M_k.
\end{equation*}
We finally introduce the \textbf{canonical half-spins}. For a rational solution $M$ of the form 
\[
    M(t,x) = M_0 + \sum_{j=1}^N \frac{A_j(t)}{x-x_j(t)} + \sum_{j=1}^N \frac{A_j^*(t)}{x-\bar x_j(t)},~ A_j(t) = s_j(t) \cdot \boldsymbol{\sigma},
\]
then, with $e_j=(\alpha_j,\beta_j)$, $\xi_j = (\beta_j,-\alpha_j)$ as in
    \[
    e_j=(\sqrt{-s_{j,1}+ \iu s_{j,2}},\sqrt{s_{j,1}+\iu s_{j,2}}),~\xi_j = (\sqrt{s_{j,1}+\iu s_{j,2}}, - \sqrt{-s_{j,1}+ \iu s_{j,2}})
    \]
the canonical half-spins are the corresponding $E_j$ and $F_j$ associated to $e_j$ and $\xi_j$ as in \eqref{Hspin}. They are indeed half-spins, as
\begin{equation}\label{Halfalp}
    E_j H F_j = \begin{pmatrix}
        \alpha_j & 0 \\
        0 & \beta_j 
    \end{pmatrix}
    \begin{pmatrix}
        1 & 1 \\
        1 & 1 \\
    \end{pmatrix}
    \begin{pmatrix}
        \beta_j & 0 \\
        0 & -\alpha_j 
    \end{pmatrix}
    = \begin{pmatrix}
        s_{j,3} & s_{j,1} - i s_{j,2} \\
        s_{j,1} + i s_{j,2} & - s_{j,3} 
    \end{pmatrix}.
\end{equation}
Now, the function $M$ and its evolution is described by its poles $x_j(t)$ and its half-spins $E_j(t)$ and $F_j(t)$. We introduce the following lemma, stating that a simple time evolution equation can be derived for $E_j$ and $F_j$.
 
\begin{lemma}\label{Evospins}
    Let $M$ be a rational function with simple poles of the form

    \[
    M(t,x) = M_0 + \sum_{j=1}^N \frac{A_j(t)}{x-x_j(t)} + \sum_{j=1}^N \frac{A_j^*(t)}{x-\bar x_j(t)},
    \]
    where $A_j(t) = s_j(t) \cdot \boldsymbol{\sigma}$ and 
    \[
    A_j(t) = E_j(t) H F_j(t),
    \]
    With 
    \[
    e_j=(\sqrt{-s_{j,1}+ \iu s_{j,2}},\sqrt{s_{j,1}+\iu s_{j,2}}),~\xi_j = (\sqrt{s_{j,1}+\iu s_{j,2}}, - \sqrt{-s_{j,1}+ \iu s_{j,2}}).
    \]
    Then, the time evolution equation describing the evolution the spins can be rewritten using the time evolution equation for the half-spins
    \[
    \dot E_j(t) = \sum_{l\neq j} \frac{E_l (\xi_j\cdot e_l)}{(x_j-x_l)^2},~ \dot F_j(t) = \sum_{l\neq j} \frac{F_l (\xi_j\cdot e_l)}{(x_j-x_l)^2}.
    \]
    Additionally, with 
    \[
    B_{j,k}=(1- \delta_{j,k}) \frac{(\xi_j \cdot e_k)}{(x_j-x_k)^2},~\dot U(t) = B(t) U(t),~U(0)=I_N,
    \]
    we have
    \[
    \mathcal{E}(t) T = [U(t)] \mathcal{E}(0) T,~\mathcal{F}(t) T = [U(t)] \mathcal{F}(0) T.
    \]
    
\end{lemma}

\begin{proof}

We recall the time evolution equation followed by $A_j(t)$

\begin{equation}\label{spinspins}
    \partial_t A_j(t) = \sum_{k\neq j} \frac{\left[ A_j(t),A_k(t) \right]}{(x_j(t)-x_k(t))^2}.
\end{equation}
Since $\partial_t (A_j(t)) = \dot E_j H F_j + E_j H \dot F_j$, (we now ommit $t$ dependency when there is no ambiguity), \eqref{spinspins} becomes

\begin{multline}\label{n}
    \dot E_j H F_j + E_j H \dot F_j \\
    = \sum_{k\neq j}^N \frac{E_j H F_j E_k H F_k - E_k H F_k E_j H F_j}{(x_j-x_k)^2} =_{(*)} 
    \sum_{k\neq j}^N \frac{E_j H F_k (\xi_j \cdot e_k)}{(x_j - x_k)^2} - \sum_{k\neq j} \frac{E_k H F_j (\xi_k \cdot e_j)}{(x_j - x_k)^2},
\end{multline}
where $(*)$ comes from lemma \ref{concat}. \eqref{n} can be rewritten as 

\begin{equation}\label{ici}
    \left( \dot E_j + \sum_{k\neq j} \frac{E_k (\xi_k \cdot e_j)}{(x_j-x_k)^2} \right) H F_j + E_j H \left( \dot F_j - \sum_{j\neq k} \frac{F_k (\xi_j \cdot e_k)}{(x_j-x_k)^2} \right)=0.
\end{equation}
Using lemma \ref{lemme:sep} on \eqref{ici}, we deduce 

\begin{equation*}
 \mathcal{K}_1(t) =    \dot F_j - \sum_{j\neq k} \frac{F_k (\xi_j \cdot e_k)}{(x_j-x_k)^2} = \begin{pmatrix}
        \dot \beta_j & 0 \\
        0 & - \dot \alpha_j
    \end{pmatrix}
    - \sum_{j\neq k} \frac{(\xi_j \cdot e_k)}{(x_j-x_k)^2} \begin{pmatrix}
        \beta_k & 0 \\
        0 & -\alpha_k
    \end{pmatrix}
    =0,
\end{equation*}
and 

\begin{equation*}
 \mathcal{K}_2(t) =    \dot E_j + \sum_{j\neq k} \frac{E_k (\xi_k \cdot e_j)}{(x_j-x_k)^2} = \begin{pmatrix}
        \dot \alpha_j & 0 \\
        0 &  \dot \beta_j
    \end{pmatrix}
    + \sum_{j\neq k} \frac{(\xi_k \cdot e_j)}{(x_j-x_k)^2} \begin{pmatrix}
        \alpha_k & 0 \\
        0 & \beta_k
    \end{pmatrix} =0.
\end{equation*}
$\mathcal{K}_1(t) = 0$ gives 
\begin{equation}\label{sys1}
\left\{
\begin{aligned}
    &\dot \beta_j = \sum_{j\neq k} \frac{(\xi_j \cdot e_k)}{(x_j-x_k)^2} \beta_k,\\
    &\dot \alpha_j = \sum_{j\neq k} \frac{(\xi_j \cdot e_k)}{(x_j-x_k)^2} \alpha_k,
    \end{aligned}
    \right.
\end{equation}
and $\mathcal{K}_2(t) = 0$ gives 

\begin{equation}\label{eqtemps}
\left\{
\begin{aligned}
    &\dot \alpha_j = \sum_{k\neq j} \frac{-(\xi_k \cdot e_j)}{(x_j-x_k)^2}\alpha_k,\\
    &\dot \beta_j = \sum_{k\neq j} \frac{-(\xi_k \cdot e_j)}{(x_j-x_k)^2}\beta_k.
    \end{aligned}
    \right.
\end{equation}
Since $(\xi_j \cdot e_k) = - (\xi_k \cdot e_j)$, te two systems \eqref{sys1} and \eqref{eqtemps} are equivalent (and compatible). The time evolution equation \eqref{eqtemps} can be rewritten as

\begin{equation}\label{systemtemps}
    \partial_t
    \begin{pmatrix}
        \alpha_1(t) \\
        \alpha_2(t) \\
        \vdots \\
        \alpha_N(t) 
    \end{pmatrix} = 
    {B}(t) \begin{pmatrix}
        \alpha_1(t) \\
        \alpha_2(t) \\
        \vdots \\
        \alpha_N(t) 
    \end{pmatrix},~
       \partial_t
    \begin{pmatrix}
        \beta_1(t) \\
        \beta_2(t) \\
        \vdots \\
        \beta_N(t) 
    \end{pmatrix} = 
    {B}(t) \begin{pmatrix}
        \beta_1(t) \\
        \beta_2(t) \\
        \vdots \\
        \beta_N(t) 
    \end{pmatrix},
\end{equation}

where $ {B}_{i,j} = \frac{\xi_i \cdot e_j}{(x_i-x_j)^2}$ for $i\neq j$, $0$ for $i=j$. Defining $\tilde U$ as 

\begin{equation*}
    \left\{
    \begin{aligned}
        &\partial_t  U(t) =  B(t)  U(t),\\
        &{U}(0) = I_N,
    \end{aligned}
    \right.
\end{equation*}
then \eqref{systemtemps} can be integrated, yielding

\begin{equation*}
    \begin{pmatrix}
        \alpha_1(t) \\
        \alpha_2(t) \\
        \vdots \\
        \alpha_N(t) 
    \end{pmatrix} = {U}(t) \begin{pmatrix}
        \alpha_1(0) \\
        \alpha_2(0) \\
        \vdots \\
        \alpha_N(0) 
    \end{pmatrix},~
    \begin{pmatrix}
        \beta_1(t) \\
        \beta_2(t) \\
        \vdots \\
        \beta_N(t) 
    \end{pmatrix} = {U}(t) \begin{pmatrix}
        \beta_1(0) \\
        \beta_2(0) \\
        \vdots \\
        \beta_N(0) 
    \end{pmatrix}.
\end{equation*}

We now justify the computation. We define $[ U]$ the doubled matrix, meaning that the coefficients $u_{i,j}$ are replaced by $2\times 2$ matrices $u_{i,j} I_{2\times 2}$. This means 

\begin{equation*}
    \left\{
        \begin{aligned}
            &[ U]_{2i,2j} =  U_{i,j}\\
            &[ U]_{2i-1,2j-1} =  U_{i,j}\\
            &[ U]_{2i,2j-1} = 0\\
            &[ U]_{2i-1,2j} = 0.
        \end{aligned}
    \right.
\end{equation*}
Now, we compute $[ U] (E_1,\dots,E_N)^T$. We have for $i=1+2r$ odd, 

\begin{multline*}
    [ U](E_1,\dots,E_N)^T_{i,1} = \sum_{j=1}^{2N} [ U]_{i,j} (E_1,\dots,E_N)_{j,1} = \sum_{ k = 0}^{N-1} [ U]_{i,2k+1} \alpha_{k+1} \\
    =\sum_{k=0}^{N-1}  U_{r,k+1} \alpha_{k+1}.
\end{multline*}
Equivalently, $[ U](E_1,\dots,E_N)^T_{i,j}$ is the scalar product (without conjugate) between two vectors
\begin{equation*}
\left\{
\begin{aligned}
    &i = 2I-1,~v_1 = (u_{I,1},0,u_{I,2},0,\dots,u_{I,N},0)  \\
    &i = 2I,~v_1 = ( 0,u_{I,1}, 0, u_{I,2},0\dots,0 , u_{I,N} ) \\
    &j = 1,~ v_2 = (\alpha_1,0,\alpha_2,0,\dots,\alpha_N,0) \\
    &j = 2,~ v_2 = (0,\beta_1,0,\beta_2,\dots,0,\beta_N)
\end{aligned}.
\right.
\end{equation*}
So if $i$ and $j$ are not of the same parity, it vanishes. If $i$ is odd $(2I-1)$ and $j$ is 1, then the coefficient is 
\begin{equation*}
    \sum_{k=1}^N u_{I,k} \alpha_k,
\end{equation*}
if $i$ is even and $j$ is $2$ then we obtain

\begin{equation*}
    \sum_{k=1}^N u_{I,k} \beta_k.
\end{equation*}
Hence, we indeed have 

\begin{equation*}
    (E_1(t),\dots,E_N(t))^T = [ U](t)  (E_1(0),\dots,E_N(0))^T. 
\end{equation*}
Since we have 
\begin{equation} \label{form}
    \mathcal{E}(t) T = \begin{pmatrix}
        E_1(t) \\
        E_2(t) \\
        \vdots \\
        E_N(t)
    \end{pmatrix},\quad 
        \mathcal{F}(t) T = \begin{pmatrix}
        F_1(t) \\
        F_2(t) \\
        \vdots \\
        F_N(t)
    \end{pmatrix},
\end{equation}
this can be rewritten as

\begin{equation*}
\boxed{
     \mathcal{E}(t) T = [ U](t) \mathcal{E}(0) T }.
\end{equation*}
Similarly, 

\begin{equation*}
\boxed{
     \mathcal{F}(t) T = [ U](t) \mathcal{F}(0) T},
\end{equation*}
which concludes the proof. 

\end{proof}

We now state and show a second lemma, providing an expression for the matrices $L$ and $B$ involved in the Lax pair, using the half-spins.

\begin{lemma}\label{Lax}
    Let $M$ be a rational solution of (HWM) of the form
    \[
    M(t,x) = M_0+ \sum_{j=1}^N \frac{A_j(t)}{x-x_j(t)} + \sum_{j=1}^N \frac{A_j^*(t)}{x-\bar x_j(t)},
    \]
    with $A_j = E_j H F_j$ and $H_j$, $F_j$ are the canonical half-spins associated to $s_j$.
    Then, defining the matrices $L$ and $B$ as 
    \[
    L_{j,k} = \delta_{j,k} \dot x_j + (1-\delta_{jk}) \frac{\xi_j \cdot e_k}{x_j-x_k},
    \]
    \[
    B_{j,k} = (1-\delta_{j,k}) \frac{\xi_j \cdot e_k}{(x_j-x_k)^2},
    \]
    $L$ and $B$ are indeed of the form
    \begin{equation}\label{L}
    L_{i,j}(t) = \delta_{j,k} \dot x_j(t) + (1-\delta_{j,k}) \varepsilon_{j,k} \frac{\sqrt{ - 2 s_j(t)\cdot s_k(t)}}{(x_j(t) - x_k(t))}
    \end{equation}
    \begin{equation}\label{B}
    B_{j,k}=(1-\delta_{j,k})\varepsilon_{j,k} \frac{\sqrt{ - 2 s_j(t)\cdot s_k(t)}}{(x_j(t) - x_k(t))^2},~ i\neq j.
    \end{equation}

    Morever, the Half-Wave maps equation in the rational case with simple poles admits the Lax pair 
    \[
    \dot L = [B,L] = BL - LB.
    \]
    The same derivation can then be made to obtain 
    \begin{equation*}
    \left\{
        \begin{aligned}
            &L(t) = U(t) L(0) U(t)^{-1},\\
            &S(t) = U(t) S(0) U(t)^{-1},\\
            &X(t) = U(t) \left( X(0) + t L(0) \right) U(t)^{-1},
        \end{aligned}
    \right. ,~\text{ where } \dot U(t) = B(t) U(t),~U(0) = I_N.
    \end{equation*}
\end{lemma}

\begin{proof}

We start with the derivation of \eqref{L} and \eqref{B}. We will use the definition of the canonical half-spins, taking profit that $e_j = (\alpha_j,\beta_j)$ and $\xi_k = (\beta_k,-\alpha_k)$.

\begin{equation*}
    s_j \cdot s_k = \frac{1}{2} Tr\left( E_j H F_j E_k H F_k  \right)= \frac{1}{2} (e_j \cdot \xi_k) (e_k \cdot \xi_j).
\end{equation*}
Now, 

\begin{equation*}
    (e_j \cdot \xi_k) = \alpha_j \beta_k - \beta_j \alpha_k = - (e_k \cdot \xi_j),
\end{equation*}
so 

\begin{equation*}
    s_j \cdot s_k = \frac{-1}{2} (\xi_k \cdot e_j)^2.
\end{equation*}

These relations imply that $\varepsilon_{j,k} \sqrt{-2 s_j \cdot s_k } = (\xi_k \cdot e_j)$, and $\xi_j \cdot e_k = \varepsilon_{k,j} \sqrt{-2 s_j\cdot s_k}$. Additionally, one can derive that $L = [B,L]$. To prove it, We compare the expression of $\dot L$ and $[B,L]$. Computing the $(j,k)$ element yields
    \[
    \left(\Dot L\right)_{j,k} = \delta_{j,k} \ddot x_j + (1-\delta_{j,k}) \left( -\frac{\xi_j\cdot e_k}{(x_j-x_k)^2} (\dot x_j-\dot x_k) + \frac{\dot \xi_j \cdot e_k + \xi_j \cdot \dot e_k}{x_j-x_k} \right),
    \]
    and 
    \[
    [B,L]_{j,k} = -(1-\delta_{j,k}) \frac{\xi_j \cdot e_k}{(x_j-x_k)^2}(\dot x_j-\dot x_k) +  \sum_{l\neq j,k}^N \frac{(\xi_j \cdot e_l) (\xi_l \cdot e_k)}{(x_j-x_l)^2(x_k-x_l)^2} (2x_l -x_j -x_k).
    \]

    For $j=k$, we obtain the time evolution equation of $\ddot x_j$, since
    \[
    \ddot x_j = 2 \sum_{l\neq k,j}^N \frac{ (\xi_j \cdot e_l) (\xi_l \cdot e_j)}{(x_l-x_j)^3} = 4 \sum_{l\neq k,j}^N \frac{(\xi_j\cdot e_l)^2}{(x_j-x_l)^3} = 4 \sum_{l\neq k,j}^N \frac{s_j\cdot s_k}{(x_j-x_l)^3}.
    \]
    For $k \neq j$, the equality reduces to
    \[
    \frac{\dot \xi_j \cdot e_k + \xi_j \cdot \dot e_k}{(x_l-x_k) - (x_l-x_j)} = \sum_{l \neq k,j} \frac{(\xi_j \cdot e_l)(\xi_l\cdot e_k)}{(x_j-x_l)^2 (x_k-x_l)^2} ((x_l - x_j) + (x_l-x_k)),
    \]
    or 
    \[
    \dot \xi_j \cdot e_k + \xi_j \cdot \dot e_k = \sum_{l \neq k,j} \frac{(\xi_j \cdot e_l)(\xi_l\cdot e_k)}{ (x_j-x_l)^2} - \sum_{l \neq k,j} \frac{(\xi_j \cdot e_l)(\xi_l\cdot e_k)}{ (x_k-x_l)^2}.
    \]
    Using now that 
    \[
    \dot F_j = \sum_{l\neq j} \frac{F_l (\xi_j\cdot e_l)}{(x_j-x_l)^2},~\dot E_j = - \sum_{k\neq j} \frac{E_k (\xi_k\cdot e_j)}{(x_j-x_k)^2},
    \]
    we can take the trace and obtain
    \[
    Tr(\dot F_j H E_k) = \dot \xi_j \cdot e_k = \sum_{l\neq j} \frac{Tr(F_l H E_k) (\xi_j \cdot e_l)}{(x_j-x_l)^2} = \sum_{l\neq j} \frac{(\xi_l \cdot e_k)(\xi_j\cdot e_l)}{(x_j-x_l)^2},
    \]
    and in a similar fashion 
    \[
    Tr(\dot E_k H F_j) = -\sum_{l\neq k}\frac{(\xi_j \cdot e_l)(\xi_l \cdot e_k)}{(x_k-x_l)^2},
    \]
    which means $\dot L=  [B,L]$ is satisfied.

\end{proof}

\subsection{Proof of theorem \ref{HST}}

This section is devoted to the proof of one of the main theorem of this paper, that we state now. It provides an explicit formula for the solution of the Half-Wave maps equation in the rational case with simple poles using constant matrices given by the condition at $t=0$. Later on, we show that this formulation is equivalent to the formula introduced in theorem \ref{gerardformula}. 

\begin{theorem}\label{HST2}
Let
\begin{equation*}
    M(t,x) = M_0 + V(t,x)
\end{equation*}
be a rational function with simple poles of the form
\begin{equation*}
    V(t,x) = \sum_{j=1}^N \frac{A_j(t)}{x-x_j(t)} + \sum_{j=1}^N \frac{A_j^*(t)}{x-\bar x_j(t)},~\im(x_j)>0,
\end{equation*}
satisfying the Cauchy problem
\begin{equation*}
    \left\{
    \begin{aligned}
        &\partial_t M(t,x) = -\frac{i}{2} [M(t,x),|\nabla| M(t,x)],\\
        &V(0,x) = V_0(x).
    \end{aligned}
    \right.
\end{equation*}
Then, for 
\begin{equation*}
    e_j = \left( \sqrt{-s_{j,1} + i s_{j,2}}, \sqrt{s_{j,1} + i s_{j,2}} \right),~ \xi_j = \left( \sqrt{s_{j,1} + i s_{j,2} }, - \sqrt{-s_{j,1} + i s_{j,2}} \right),
\end{equation*}
and $e_j=(\alpha_j,\beta_j)$, we correspondingly define 

\begin{equation*}
    E_j = \begin{pmatrix}
        \alpha_j & 0 \\
        0 & \beta_j 
    \end{pmatrix},~ F_j = \begin{pmatrix}
        \beta_j & 0 \\
        0 & -\alpha_j 
    \end{pmatrix},
\end{equation*}

and $\mathcal{E}$ and $\mathcal{F}$ as in \eqref{Erond} and \eqref{Frond}. 
Then, $\Pi_- V(t,x)$ is given by, with $T$ and $\mathcal{H}$ defined as in \eqref{constants}, 

\begin{equation*}
    \Pi_{-} V(t,x) = -T^T \mathcal{E}(0) \mathcal{H} [X(0) + tL(0) - x I_N]^{-1} \mathcal{F}(0) T.
\end{equation*}

\end{theorem}

\begin{proof}
The strategy of this proof will be to express $\Pi_{-} V(t,x)$ using our formalism. Then, combining the time evolution of the poles in the moving frame and the time evolution of the half-spins will provide the result, as those evolutions almost entirely cancel out.

    Let $M$ be a rational solution of (HWM) of the form
    \[
    M(t,x) = M_0+ \sum_{j=1}^N \frac{A_j(t)}{x-x_j(t)} + \sum_{j=1}^N \frac{A_j^*(t)}{x-\bar x_j(t)},
    \]
    with $A_j = E_j H F_j$ and $H_j$, $F_j$ are the canonical half-spins associated to $s_j$.
With $X(t) = diag(x_1,\dots,x_N)$, and 
\begin{equation}\label{LB}
    L_{i,j}(t) = \left\{
    \begin{aligned}
    & \dot x_i(t),~ i=j,\\
    &\frac{ \xi_i  \cdot e_j }{(x_i(t) - x_j(t))},~ i\neq j,
    \end{aligned}
    \right. 
    ~
    B_{i,j}(t) = \left\{
    \begin{aligned}
    & 0,~ i=j,\\
    & \frac{ \xi_i \cdot e_j }{(x_i(t) - x_j(t))^2},~ i\neq j,
    \end{aligned}
    \right.
\end{equation}
then Lemma \ref{Lax} provides in particular with $\partial_t U(t) = B(t) U(t)$ and $U(0)= I_N$,
\begin{equation}\label{poles}
    L(t) = U(t) L(0) U(t)^{-1},~ X(t) = U(t) (X(0) + t L(0)) U(t)^{-1}.
\end{equation}

We consider $[X] \in \mathbb{C}^{2n \times 2n}$, the doubled matrix of $X$ defined as in \eqref{double}. \eqref{poles} gives in particular, using appendix \ref{commute},
\begin{equation}\label{23}
    [X](t) = [U](t) ([X](0) + t [L](0)) [U]^{-1}(t),
\end{equation}

Using \eqref{Halfalp}, we have
\begin{equation*}
    \mathcal{E}(t) \mathcal{H} \mathcal{F}(t) = \begin{pmatrix}
        A_1(t) & 0 & \dots & 0\\
        0& A_2(t) & \dots & 0 \\
        \dots & \dots & \dots & \dots \\
        0& 0 & \dots & A_N(t)
    \end{pmatrix}.
\end{equation*}
Similarly, 
\begin{equation*}
    - \mathcal{E}(t) \mathcal{H} [X(t)-xI_N]^{-1} \mathcal{F}(t) = \begin{pmatrix}
        \frac{A_1(t)}{x-x_1(t)} & 0 & \dots & 0\\
        0& \frac{A_2(t)}{x-x_2(t)} & \dots & 0 \\
        \dots & \dots & \dots & \dots \\
        0& 0 & \dots & \frac{A_N(t)}{x-x_N(t)}
    \end{pmatrix}.
\end{equation*}
Now, we would like to state that
\begin{equation*}
    Tr\left( - \mathcal{E}(t) \mathcal{H} [X(t)-xI_N]^{-1} \mathcal{F}(t) \right) = \sum_{j=1}^N \frac{A_j(t)}{x-x_j(t)} = \Pi_{-} M(t,x),
\end{equation*}
but since $Tr(A_j(t)) = 0$, we have $Tr\left( - \mathcal{E}(t) \mathcal{H} [X(t)-xI_N]^{-1} \mathcal{F}(t) \right) = 0$. Instead, the following identity holds 
\begin{equation}\label{inter}
    -T^T \mathcal{E}(t) \mathcal{H} [X(t)-xI_N]^{-1} \mathcal{F}(t) T = \sum_{j=1}^N \frac{A_j(t)}{x-x_j(t)} = \Pi_- M(t,x).
\end{equation}
Now, we use the expression \eqref{23} obtained using the Lax structure to simplify \eqref{inter}. First, we write

\begin{equation*}
    [X(t) - x I_N] = [U](t) \left( [X(0)] + t [L(0)] - x [I_N] \right) [U]^{-1}(t)
\end{equation*}
Since $U^T = U^{-1}$, we also have $[U]^T = [U]^{-1}$, and we also have that $[U]$ (a doubled matrix) commutes with $\mathcal{H}$ (a constant diagonal matrix with blocks) using lemma \ref{commute}, so

\begin{multline*}
    \Pi_- M(t,x) = -T^T \mathcal{E}(t) \mathcal{H} [U](t) [X(0)+t L(0) - x I_N]^{-1} [U](t)^{-1} \mathcal{F}(t) T \\
    = - (\mathcal{H} [U]^{-1} \mathcal{E}(t) T)^T [X(0)+ t L(0) - x I_N]^{-1} [U]^{-1} \mathcal{F}(t) T,
\end{multline*}
which means 
\begin{equation*}
\boxed{ \Pi_- M(t,x) = - (\mathcal{H} [U]^{-1} \mathcal{E}(t) T)^T [X(0)+ t L(0) - x I_N]^{-1} [U]^{-1} \mathcal{F}(t) T }.
\end{equation*}
Using Lemma \ref{Evospins}, we obtain
\begin{equation}\label{objective}
    [U]^{-1} \mathcal{E}(t) T = \mathcal{E}(0) T,~ [U]^{-1} \mathcal{F}(t) T = \mathcal{F}(0) T,
\end{equation}
which gives 
\begin{equation*}
    \Pi_- M(t,x) = - (\mathcal{H} \mathcal{E}(0) T)^T [X(0) + t L(0) - x I_N]^{-1} \mathcal{F}(0) T,
\end{equation*}
so

\begin{equation*}
\boxed{
    \Pi_- M(t,x) = - T^T \mathcal{E}(0) \mathcal{H} [X(0)+tL(0)-xI_N]^{-1} \mathcal{F}(0)T },
\end{equation*}
which is our desired result.

\end{proof}

\subsection{Equivalence between the formulas}

This section is devoted to the proof that the formula provided by theorem \ref{HST2} and theorem \ref{gerardformula} are equivalent. We first start with the following lemma to go from $\Pi_-$ to $\Pi_+$, which holds by conjugation.

\begin{lemma}
    Let 
    \[
    M(t,x) = M_0 + V(t,x),
    \]
    with
    \[
    V(t,x) =\sum_{j=1}^N \frac{A_j(t)}{x - x_j(t)} + \sum_{j=1}^{N} \frac{A_j^*(t)}{x-\bar x_j(t)},~ ~ \left\{ \begin{aligned}
    &A_j(t) = E_j(t) H F_j(t),\\
    &\im x_j(t) >0.
    \end{aligned}
    \right.
    \]
    be a solution of (HWM).

    Then, the two formulas are equivalent
    \begin{equation*}
        \begin{aligned}
            &\Pi_- V(t,x) = - T^T \mathcal{E}_0 \mathcal{H} \left[ X(0) + t L(0) - xI_N\right]^{-1} \mathcal{F}_0 T \\
            \Leftrightarrow & \Pi_+ V(t,x) = -T^T \mathcal{F}_0^* \left[ \bar X(0) + t L^*(0) - xI_N \right]^{-1} \mathcal{H} \mathcal{E}_0^* T.
        \end{aligned}
    \end{equation*}
    
\end{lemma}

\begin{proof}
    To show this, we simply write that in our case, 
    \[
    \Pi_- V(t,x) = \left( \Pi_+ V(t,x) \right)^*.
    \]
    We hence show that the expressions $- T^T \mathcal{E}_0 \mathcal{H} \left[ X(0) + t L(0) - xI_N\right]^{-1} \mathcal{F}_0 T$ and \\
    $-T^T \mathcal{F}_0^* \left[ \bar X(0) + t \mathcal{L}(0) - xI_N \right]^{-1} \mathcal{H} \mathcal{E}_0^* T$ are hermitian conjugates of each other.

    Since $T^* = T^T$ and $\mathcal{H}^*=\mathcal{H}$, we obtain

    \begin{multline*}
        \left( -T^T \mathcal{E}_0 \mathcal{H} [X(0)+tL(0)-x I_N]^{-1} \mathcal{F}_0 T \right)^* \\
        = - T^T \mathcal{F}_0^*  [\left(X(0)+tL(0)-xI_N\right)^{*}]^{-1}  \mathcal{H} \mathcal{E}_0^* T \\
        = -T^T \mathcal{F}_0^* \left[ \bar X(0)+t L^*(0) - xI_N \right]^{-1} \mathcal{H} \mathcal{E}_0^* T.
    \end{multline*}

\end{proof}

Note that since we have 

\begin{equation*}
    \mathcal{E}_0^* = \Diag{E_1(0)^*}{E_2(0)^*}{E_N(0)^*},~ \mathcal{F}_0^* = \Diag{F_1(0)^*}{F_2(0)^*}{F_N(0)^*},
\end{equation*}

and since $A_j^* = (E_j H F_j)^* = F_j^* H E_j^*$, the matrices $\mathcal{F}_0^*$ and $\mathcal{E}_0^*$ are simply the matrices of the half spins of the spins $A_j^*$ at $t=0$. 

We now state the main theorem of this section.

\begin{theorem}
    Let 
    \[
    M(t,x) = M_0 + V(t,x),
    \]
    with
    \[
    V(t,x) =\sum_{j=1}^N \frac{A_j(t)}{x - x_j(t)} + \sum_{j=1}^{N} \frac{A_j^*(t)}{x-\bar x_j(t)},~ ~ \left\{ \begin{aligned}
    &A_j(t) = E_j(t) H F_j(t),\\
    &\im \left(x_j(t)\right) >0.
    \end{aligned}
    \right.
    \]
    be a solution of (HWM).

    Then, with $V_0(y) = V(0,y)$

    \[
    \Pi_+ V(t,x) = \frac{1}{2 \iu \pi} I_+\left( (G - t T_{U_0} - xI_N)^{-1} \Pi_+ V_0 \right).
    \]
    
\end{theorem}

\begin{proof}
    We already have that 
    \[
    \Pi_+ V(t,x) = - T^T \mathcal{F}^*_0 \left[ \bar X(0) + t L^*(0) - xI_N \right]^{-1} \mathcal{H} \mathcal{E}^*_0 T.
    \]

We now show that this is equivalent to 

\begin{equation*}
    \Pi_+V(t,x) = \frac{1}{2\iu \pi}I_+\left( (G+T_{U_0}-xI)^{-1} \Pi_+ V(0) \right).
\end{equation*}

We have 
\begin{equation*}
    \Pi_+ V_0 = \sum_{j=1}^N \frac{ A_j^*(0)}{y-\bar x_j} = \sum_{j=1}^N \frac{F_j^* H E_j^*}{y - \bar x_j}.
\end{equation*}

We now consider the basis

\[
\mathcal{B}=(K_1,\dots,K_N),~K_1 = \frac{F_1^* H}{y-\bar x_1},\dots,K_N=\frac{F_N^{*}H }{y-\bar x_N},
\]

and define the decomposition in this basis, for $(C_1,\dots,C_N)$, $N$ $2\times2$ complex and diagonal matrices, as

\begin{equation*}
    \left( C_1,\dots,C_N \right)_\mathcal{B} = \sum_{j=1}^N K_j C_j = \sum_{j=1}^N \frac{F_j^* H C_j}{y-\bar x_j},
\end{equation*}

finally, we define $V_\mathcal{B}=\operatorname{span}(\mathcal{B})$ as 
\[
V_\mathcal{B} = \left\{ f,~f=(C_1,\dots,C_N)_{\mathcal{B}},~C_1,\dots,C_N \in \mathbb{C}^{2\times 2},\text{ diagonal.} \right\}
\]

We have that $\Pi_+ V_0 \in V_\mathcal{B}$ and

\[
 \Pi_+ V_0  = \left(E_1^*,\dots,E_N^* \right)_\mathcal{B}.
\]

Now, in $\mathcal{B}$, the operator $G$ is represented by the matrix 
\begin{equation*}
    G=[\bar X(0)],~\bar X(0) = \Diag{\bar x_1(0)}{\bar x_2(0)}{\bar x_N(0)}.
\end{equation*}

Let us now consider a function $f: \mathbb{R} \to \mathbb{C}^2$, such that $f\in V_\mathcal{B}$ and 

\[
f = (C_1,\dots,C_N)_\mathcal{B}.
\]

Then,

\[
\frac{1}{2 \iu \pi} I_+(f) = \frac{1}{2 \iu \pi} I_+ \left( \sum_{j=1}^N \frac{F_j^* H C_j }{y-\bar x_j} \right) = - \sum_{j=1}^N F_j^* H C_j.
\]

Hence, if $F$ is a diagonal matrix representing $f$ in $\mathcal{B}$, we have
\[
F = \Diag{C_1}{C_2}{C_N},~ \frac{1}{2 \iu \pi}  I_+ (f) = -T^T \mathcal{F}^*_0 \mathcal{H} F T.
\]

We now find a representation of the operator $T_{U_0}$, which is linear. 
\[
\begin{aligned}
&T_{U_0} \left( K_j C_j \right) = T_{U_0} \left( \frac{F_j^* H C_j}{x-\bar x_j} \right) \\
&= \left( M_0 + \sum_{k\neq j} \frac{A_k^*}{\bar x_j - \bar x_k} + \sum_{k=1}^N \frac{A_k}{\bar x_j - x_k} \right) \frac{F_j^* H C_j}{x - \bar x_j} - \sum_{k\neq j} \frac{A_k^* F_j^* H C_j}{(\bar x_j - \bar x_k)(y-\bar x_k)}\\
&= \frac{B_J^* F_j^* H C_j}{y-\bar x_j} + \sum_{k\neq j}\frac{F_k^* H E_k^* F_j^* H C_j}{(\bar x_k - \bar x_j)(x-\bar x_k)}\\
&=- \bar b_j K_j C_j + \sum_{k\neq j} K_k  \frac{(\bar e_k\cdot \bar \xi_j) C_j}{\bar x_k - \bar x_j} = \sum_{k=1}^N \mathcal{L}_{k,j} K_k C_j,\\
&\text{With }\mathcal{L}_{k,j} = \frac{(\bar e_k\cdot \bar \xi_j)}{\bar x_k - \bar x_j},~k\neq j,~ -\bar b_j,~ k=j.
\end{aligned}
\]

Since no $2\times 2$ product is involved in $\mathcal{L}_{k,j} \in \mathbb{C}$, we can then represent $T_{U_0}$ by a doubled matrix. Indeed, we have then by linearity 

\[
T_{U_0} \left( \sum_{j=1}^N K_j C_j \right) = \sum_{j=1}^N \sum_{k=1}^N \mathcal{L}_{k,j} K_k C_j,
\]

which means that $T_{U_0}$ is represented in the basis $\mathcal{B}$ by the doubled matrix $[\mathcal{L}]$, since 

\[
\begin{pmatrix}
    \LL_{1,1} & 0         & \LL_{1,2} &         0 & \dots & \dots & \LL_{1,N} & 0        \\
    0         & \LL_{1,1} & 0         & \LL_{1,2} & \dots & \dots & 0         & \LL_{1,N}\\
    \LL_{2,1} & 0         & \LL_{2,2} & 0         & \dots & \dots & \vdots & \vdots \\
    0         & \LL_{2,1} & 0         & \LL_{2,2} & \dots & \dots & \vdots & \vdots \\
    \vdots    & \vdots    & \vdots    & \vdots    & \ddots & \ddots & \vdots & \vdots \\
    \LL_{N,1} & 0         & \dots     & \dots     & \dots     & \dots & \LL_{N,N} & 0 \\
    0         & \LL_{N,1} & \dots     & \dots     & \dots     & \dots & 0 & \LL_{N,N} \\
\end{pmatrix} 
\begin{pNiceArray}{cc}
    \Block{2-2}<\Large>{C_1} \\
    \\
    \Block{2-2}<\Large>{C_2} \\
    \\
    \vdots & \vdots \\
    \vdots & \vdots \\
    \Block{2-2}<\Large>{C_N} \\
    \\
\end{pNiceArray}=
\begin{pmatrix}
    \left( \sum_{k=1}^N \mathcal{L}_{1,k} C_k \right)  \\
    \vdots  \\
    \left( \sum_{k=1}^N \mathcal{L}_{N,k} C_k \right) 
\end{pmatrix}
\]

Hence, if $f$ is represented by $\mathcal{C}=(C_1,\dots,C_N)_\mathcal{B}$ in $\mathcal{B}$, then $T_{U_0}(f)$ is represented by $[\LL] \mathcal{C} $ in $\BB$. Note that we have $\mathcal{L} = -L^*$, since 
\[
\left( \mathcal{L}^* \right)_{k,j} = \frac{e_j \cdot \xi_k}{x_j-x_k} = - \frac{\xi_j \cdot e_k}{x_j-x_k} = - \left( L \right)_{k,j},~\text{for } k\neq j,
\]
and 
\[
(\mathcal{L})^*_{j,j} = - b_j = - L_{j,j}.
\]
Finally, $\Pi_+V(y,0) = (E_1^*(0),\dots,E_N^*(0))_\mathcal{B}$ in $\mathcal{B}$ gives 
\[
\begin{aligned}
(G-t T_{U_0}-xI_N)^{-1}(\Pi_+ V_0) &= \left( [\bar X(0) - t \mathcal{L} - xI_N]^{-1} \left( E_1^*(0),\dots,E_N^*(0) \right)^T \right)_\BB \\
& = \left( [\bar X(0) + t L^* - xI_N]^{-1} \mathcal{E}^*_0 T\right)_\BB.
\end{aligned}
\]
Hence, 
\[
\frac{1}{2\iu \pi} I_+ \left( (G-t T_{U_0}-xI_N)^{-1}(\Pi_+ V_0 \right) = - T^T \mathcal{F}_0^* \mathcal{H} [\bar X(0)+ t L^* - x I_N]^{-1} \mathcal{E}_0^* T, 
\]

which concludes the proof that the two formulas are equivalent. 

\end{proof}

\printbibliography

\appendix

\section{Technical results}

\begin{lemma}\label{concat}
    We have $H D_1 D_2 H = (d_1 \cdot d_2 ) H $ when $D_1$, $D_2$ are two diagonal matrices with $d_1$ and $d_2$ on the diagonal. 
\end{lemma}

\begin{proof}
\begin{multline*}
    \begin{pmatrix}
        1 & 1 \\
        1 & 1 \\
    \end{pmatrix}
    \begin{pmatrix}
        d_1(1) & 0 \\
        0 & d_1(2) 
    \end{pmatrix}
    \begin{pmatrix}
        d_2(1) & 0 \\
        0 & d_2(2) 
    \end{pmatrix}
    \begin{pmatrix}
        1 & 1 \\
        1 & 1 \\
    \end{pmatrix}
    \\
    = \begin{pmatrix}
        1 & 1 \\
        1 & 1 \\
    \end{pmatrix}
    \begin{pmatrix}
        d_1(1)d_2(1) & d_1(1)d_2(1) \\
        d_1(2)d_2(2) & d_1(2)d_2(2) \\
    \end{pmatrix}
    =  (d_1(1) d_2(1) + d_1(2) d_2(2)) H.
\end{multline*}
\end{proof}

\begin{lemma}\label{lemme:sep}
    Let 
    \begin{equation*}
        \mathcal{K}_1 = \begin{pmatrix}
            \gamma & 0 \\
            0 & \delta
        \end{pmatrix},~
        \mathcal{K}_2 = \begin{pmatrix}
            \delta & 0 \\
            0 & -\gamma
        \end{pmatrix},
    \end{equation*}

    then 
    \begin{equation*}
        \mathcal{K}_1 H F_j + E_j H \mathcal{K}_2 = 0
    \end{equation*}
    implies $\mathcal{K}_1 = \mathcal{K}_2 = 0$.

\end{lemma}
\begin{proof}

    We obtain from $\mathcal{K}_1 H F_j + E_j H \mathcal{K}_2 = 0$ that 
    \begin{equation*}
        \begin{pmatrix}
            \gamma \beta_j & - \gamma \alpha_j \\
            \delta \beta_j & - \delta \alpha_j
        \end{pmatrix} + 
        \begin{pmatrix}
            \delta \alpha_j & - \gamma \alpha_j \\
            \delta \beta_j & - \gamma \beta_j
        \end{pmatrix} =   
        \begin{pmatrix}
            0 &  0 \\
            0 & 0
        \end{pmatrix}.
    \end{equation*}

    Assume first $\alpha_j \neq 0$, then $\gamma = 0$, and $\delta \alpha_j + \gamma \beta_j = 0$ gives $\delta = 0$. If $\beta_j \neq 0$, then $\delta = 0$, and $\delta \alpha_j + \gamma \beta_j = 0$ gives $\gamma = 0$.
    
\end{proof}

\begin{lemma}\label{commute}
    Let $A$ and $B$ be two $N\times N$ matrices with complex entries. Then, 
    \begin{equation}\label{premiercaslemme}
        [AB] = [A]\cdot [B],
    \end{equation}
    and with $C$ a $2 \times 2$ matrix, 
    \begin{equation}\label{deuxiemecaslemme}
        \begin{pmatrix}
        C & 0_{2\times 2} & \dots & 0_{2\times 2} \\
        0_{2\times 2} & C & \dots & 0_{2\times 2} \\
        \vdots & \vdots & \ddots & \vdots \\
        0_{2\times 2} & 0_{2\times 2} & \dots & C
        \end{pmatrix} \cdot A = A \cdot \begin{pmatrix}
        C & 0_{2\times 2} & \dots & 0_{2\times 2} \\
        0_{2\times 2} & C & \dots & 0_{2\times 2} \\
        \vdots & \vdots & \ddots & \vdots \\
        0_{2\times 2} & 0_{2\times 2} & \dots & C
        \end{pmatrix}.
    \end{equation}
    
\end{lemma}

\begin{proof}
    The results are easily obtained by matrix multiplication. For \eqref{premiercaslemme}, we have 
    \begin{equation*}
        [AB] = \left( \sum_{k=1}^N a_{i,k} b_{k,j} I_{2\times 2} \right)_{i,j} = \left( \sum_{k=1}^N a_{i,k} I_{2\times 2} b_{k,j} I_{2\times 2} \right)_{i,j}=[A]\cdot [B]
    \end{equation*}

    For \eqref{deuxiemecaslemme}, we have with 
    
    \begin{equation*}
        \mathcal{C} = \begin{pmatrix}
        C & 0_{2\times 2} & \dots & 0_{2\times 2} \\
        0_{2\times 2} & C & \dots & 0_{2\times 2} \\
        \vdots & \vdots & \ddots & \vdots \\
        0_{2\times 2} & 0_{2\times 2} & \dots & C
        \end{pmatrix},
    \end{equation*}
\end{proof}

the identity 
\begin{equation*}
    \mathcal{C} \cdot [A] = \begin{pmatrix}
        a_{1,1} C & a_{1,2} C & \dots & a_{1,N} C \\
        a_{2,1} C & a_{2,2} C & \dots & a_{2,N} C \\
        \vdots & \vdots & \ddots & \vdots \\
        a_{N,1} C & a_{N,2} C & \dots & a_{N,N} C
        \end{pmatrix} = [A] \cdot \mathcal{C}.
\end{equation*}

\section{Conventions and $2\times2$ formulation}
\textbf{\large Time evolution and constraints}\\
We look at a solution of the (HWM) in the rational form with simple poles
\[
M(t,x) = M_0 + \sum_{j=1}^N \frac{A_j(t)}{x-x_j(t)} + \sum_{j=1}^N \frac{A_j^*(t)}{x-\bar x_j(t)},~ \im \left(x_j(t)\right) >0.
\]
$M$ solves (HWM) if and only if 
\[
\partial_t M(t,x) = - \frac{\iu}{2} [M(t,x),|\nabla| M(t,x)]
\]
We have for $\partial_t M(t,x)$
\[
\partial_t M(t,x) = \sum_{j=1}^N \left( \frac{A_j'(t)}{x-x_j(t)} + \frac{x_j'(t) A_j(t)}{(x-x_j'(t))^2} \right) + \sum_{j=1}^N \left( \frac{\left(A_j'\right)^*(t)}{x-\bar x_j(t)} + \frac{\bar x_j'(t) A_j(t)}{(x-\bar x_j'(t))^2} \right),
\]
and for $|\nabla| M(t,x)$,
\[
|\nabla| U = -\iu \sum_{j=1}^N \frac{-A_j(t)}{\left(x-x_j(t)\right)^2} - \iu \sum_{j=1}^N \frac{A_j^*(t)}{\left(x-\bar x_j(t)\right)^2}.
\]
We can rewrite the right-hand side as
\[
\begin{aligned}
-\frac{\iu}{2} [M(t,x)|\nabla| M(t,x)] &=\\
\frac{1}{2}\Bigg[ U_\infty + \sum_{k=1}^N \frac{A_k}{x-x_k}&+\sum_{k=1}^N \frac{A_k^*}{x-\bar x_k},~ \sum_{j=1}^N \frac{A_j}{(x-x_j)^2} - \sum_{j=1}^N \frac{A_j^*}{(x-\bar x_j)^2} \Bigg] 
\end{aligned}
\]
\[
 = \frac{1}{2} \sum_{j=1}^N \frac{1}{(x-x_j)^2} [U_\infty,A_j] - \frac{1}{2} \sum_{j=1}^N \frac{1}{(x-\bar x_j)^2} [U_\infty,A_j^*]
\]
\[
+ \frac{1}{2} \sum_{k=1}^N \sum_{j\neq k}^N \frac{1}{(x_k-x_j)^2} \left( \frac{1}{x-x_k} - \frac{1}{x-x_j} \right) [A_k,A_j] + \frac{1}{x_j-x_k} \frac{1}{(x-x_j)^2} [A_k,A_j]
\]
\[
+ \frac{1}{2} \sum_{k=1}^N \sum_{j=1}^N \frac{1}{(\bar x_k-x_j)^2} \left( \frac{1}{x-\bar x_k} - \frac{1}{x-x_j} \right) [A_k^*,A_j] + \frac{1}{ x_j-\bar x_k} \frac{1}{(x-x_j)^2} [A_k^*,A_j]
\]
\[
- \frac{1}{2} \sum_{k=1}^N \sum_{j=1}^N \frac{1}{(x_k-\bar x_j)^2} \left( \frac{1}{x-x_k} - \frac{1}{x-\bar x_j} \right) [A_k,A_j^*] - \frac{1}{\bar x_j- x_k} \frac{1}{(x-\bar x_j)^2} [A_k,A_j^*]
\]
\[
- \frac{1}{2} \sum_{k=1}^N \sum_{j\neq k}^N \frac{1}{(\bar x_k-\bar x_j)^2} \left( \frac{1}{x-\bar x_k} - \frac{1}{x-\bar x_j} \right) [A_k^*,A_j^*] - \frac{1}{\bar x_j- \bar x_k} \frac{1}{(x-\bar x_j)^2} [A_k^*,A_j^*]
\]
Regrouping the term, we obtain
\[
\begin{aligned}
\frac{1}{(x-x_j)^2}:~& \frac{1}{2} \left( [U_\infty,A_j] + \sum_{k\neq j} \left[ \frac{A_k}{x_j-x_k},A_j \right]+\sum_{k=1}^N \left[ \frac{A_k^*}{x_j-\bar x_k},A_j \right]\right)\\
=& \frac{1}{2} \left[ U_\infty + \sum_{k\neq j} \frac{A_k}{x_j-x_k} + \sum_{k=1}^N \frac{A_k^*}{x_j-\bar x_k},~ A_j \right]
\end{aligned}
\]
equaling with the left-hand-side yields
\[
\left(\partial_t x_j \right)A_j = \frac{1}{2} \left[ U_\infty + \sum_{k\neq j} \frac{A_k}{x_j-x_k} + \sum_{k=1}^N \frac{A_k^*}{x_j-\bar x_k},~ A_j \right]= \frac{1}{2}\left[B_j,A_j \right].
\]
Considering now the terms in $(x-x_j)^{-1}$ 
\[
\begin{aligned}
    \frac{1}{(x-x_j)}:~& \frac{1}{2} \left(\sum_{k\neq j} \frac{[A_j,A_k]}{(x_j-x_k)^2} - \sum_{k\neq j} \frac{[A_k,A_j]}{(x_k-x_j)^2} - \sum_{k=1}^N \frac{[A_k^*,A_j]}{(\bar x_k-x_j)^2} - \sum_{k=1}^N \frac{[A_j,A_k^*]}{(x_j-\bar x_k)^2}\right) \\
    &= \sum_{k\neq j} \frac{[A_j,A_k]}{(x_j-x_k)^2}.
\end{aligned}
\]
Since the term in $(x-x_j)^{-1}$ is just $\partial_t A_j$ in the left-hand-side, we get
\[
\partial_t A_j = \sum_{k\neq j} \frac{[A_j,A_k]}{(x_j-x_k)^2}.
\]
In addition, writing $M(t,x)^2=I_{2}$ gives in particular
\[
A_j B_j + B_j A_j = 0.
\]
We then obtain 
\[
\partial_t x_j A_j = B_j A_j = b_j A_j,~ b_j = \partial_t x_j.
\]
\textbf{\large Additional formulas}

\[
\frac{1}{x-x_j} \frac{1}{x-x_k} = \frac{1}{x_j-x_k} \left( \frac{1}{x-x_j} - \frac{1}{x-x_k} \right),
\]

\[
\frac{1}{x-x_j} \frac{1}{(x-x_k)^2} = \frac{1}{(x_j-x_k)^2 (x-x_j)} - \frac{1}{(x_j-x_k)^2 (x-x_k)} - \frac{1}{(x_j-x_k) (x-x_k)^2}.
\]
The condition coming from \cite{matsuno2022integrability}, adapted to our notation, reads

\[
\dot x_j(t) = \frac{s_j\times s_j^*}{s_j \cdot s_j^*} \cdot \iu \left( m_0 + \sum_{k\neq j}^N \frac{s_k}{x_j-x_k} + \sum_{k=1}^N \frac{s_k^*}{x_j-\bar x_k} \right).
\]
Using the two-by-two matrices with our conventions, it becomes

\[
\dot x_j(t) = \frac{\frac{1}{2} Tr\left( \frac{1}{2 \iu} [A_j,A_j^*] \iu B_j \right)}{\frac{1}{2} Tr(A_j A_j^*)} = \frac{1}{2}\frac{ Tr\left(  [A_j,A_j^*]  B_j \right)}{ Tr(A_j A_j^*)}.
\]
If $B_j A_j = b_j A_j$ and $B_j^* A_j* = -\bar b_j A_j^*$, then it becomes

\[
\dot x_j(t) = \frac{1}{2} \frac{ Tr(b_j A_j A_j^*) - Tr(B_j^* A_j^* A_j)^* }{Tr(A_j A_j^*)} =  \frac{b_j Tr(A_j A_j^*)}{Tr(A_j A_j^*)} =  b_j.
\]

\end{document}